\documentclass[12pt]{article}

\usepackage[T1]{fontenc}
\usepackage[left=25mm, right=25mm, top=25mm, bottom=25mm]{geometry}
\usepackage{amssymb, amsmath, amscd, amsthm, verbatim, 
	mathrsfs, tikz, color , amsfonts}
\usepackage{cite}
\usepackage{setspace} 
\allowdisplaybreaks
\usepackage[vcentermath, enableskew, noautoscale]{youngtab}
\usepackage{young}

\usepackage{hyperref}

\usepackage{bm}
\usepackage{enumerate}
\usepackage{graphicx}
\usepackage{mathptmx}
\usepackage{enumitem} 
\usepackage{marvosym}    
\usepackage{hyphenat}

\newtheorem{remark}{Remark}[section]
\newtheorem{theorem}{Therorem}[section]
\newtheorem{definition}{Definition}[section]
\newtheorem{lemma}{Lemma}[section]

\title{\textbf{Bi-shadowing of Quasi-semi-hyperbolic Pseudo-orbit}}
\author{
    Yan He$^{\href{mailto:hy2365446496@163.com}{\textnormal{\Letter}}}$
    \thanks{Email: {202306021013@stu.cqu.edu.cn}}, 
    Meihua  Jin$^{\href{mailto:202206021016@stu.cqu.edu.cn}{\textrm{\Letter}}}$
 \thanks{Corresponding author, Email:202206021016@stu.cqu.edu.cn}
    \\
    \small \textit{College of Mathematics and Statistics, Chongqing University, Chongqing 401331, China}
}

\date{}

\begin{document}
	\hbadness=10000
	\hfuzz=0.2pt  
	\overfullrule=0pt  
	\bibliographystyle{plain}
	\maketitle
	\renewcommand{\abstractname}{}
	\begin{abstract}
		\noindent\textbf{Abstract:} In this paper, we introduce the
		concept of quasi-semi hyperbolic pseudo-orbits and prove that
		quasi-semi hyperbolicity implies quasi hyperbolicity provided the
		error magnitude are sufficiently small. We also have successively
		demonstrated that both finite quasi-hyperbolic pseudo-orbits and
		infinite quasi-semi hyperbolic pseudo-orbits possess the
		bi-shadowing property, and thus we establish the periodicity.
	\end{abstract}

\section{Introduction}

\hspace{1em}Differential dynamical systems mainly focus on the stability and variation characteristics exhibited by the system under disturbances. Among them, the invariant properties include $ \Omega $-stability and structural stability, while shadowing property is an important indicator for describing the stability of the system. The concept of pseudoorbit is widely applied in various fields of power systems, especially playing a core role in qualitative analysis. The pseudoorbit shadowing theory is an important tool in the study of dynamical systems and also the theoretical basis of computer numerical simulation. It is particularly important in analyzing the stability of differential dynamical systems, and also provides profound insights for exploring the topological properties, traversal properties and statistical characteristics of the system.

Historically, Anosov\cite{DA} and Bowen\cite{RB} first proposed the shadowing lemma for general diffeomorphisms: that is, if $ \Lambda $ is a hyperbolic set of diffeomorphisms $ f $, then there exists a neighborhood $ U $ of $ \Lambda $ such that $ f $ has pseudoorbit shadowing properties on $ U $. It marks the beginning of shadowing research in power systems. Later, Eirola, Nevanlinna and Pilyugin investigated the limit shadowing property of differential homeomorphisms and the $ L^{P} $ shadowing properties of hyperbolic sets\cite{TO}. Hirsch investigated the shadowing property\cite{MH}, and Pilyugin proved that structurally stable differential homeomorphisms possess Lipschitz pseudoorbit shadowing property\cite{SP}. In many situations, because of the existence of singularities, the results on flows are not the trivial translations of ones on discrete case(diffeomorphisms) (e.g. see \cite{SG,SD}). Building upon the research insights of Liao \cite{LS} on the shadowing lemma for quasi-hyperbolic orbit arcs,  Gan established a generalized shadowing lemma\cite{SG}. Zhang and Zhou extended the concept of quasi-hyperbolic pseudoorbits to quasi-partially hyperbolic pseudoorbits, and established quasi-shadowing properties and limit quasi-shadowing properties \cite{ZF}. 

Compared to autonomous discrete dynamical systems, non-autonomous discrete dynamical systems exhibit more complex dynamical behavior. Currently, academic research focuses on multiple aspects, including but not limited to chaotic properties, topological entropy, and sensitivity in non-autonomous dynamical systems. For example, Anosov families
provide a framework where robust analytical techniques have been developed to study these systems(see \cite{LB,JM,WC} and references therein).

The concepts of semi-hyperbolicity and bi-shadowing were first cited and proved by Diamond, Kloeden \cite{PD1,PD2,PD3}. The main result established is that semi-hyperbolicity is sufficient to guarantee that a dynamical system generated by a Lipschitz map possesses the bi-shadowing property with respect to a class of perturbed systems generated by continuous maps. Mazur et al. indicated that for $ C^{1} $
diffeomorphism, the semi-hyperbolicity of invariant sets implies its hyperbolicity, and provided some exact estimates of hyperbolic constants by semi-hyperbolic constants \cite{MM}.
Besides, Chen and Zhou proved that the semi-partially hyperbolic families also possesses the bi-quasi-shadowing property.the existence of bi-quasi-shadowing property\cite{LC}

In this paper, we extend the notion of quasi-hyperbolic pseudo-orbits to a more general setting by introducing the concepts of quasi-semi hyperbolic pseudo-orbits. We demonstrate that a quasi-semi hyperbolic pseudo-orbit is quasi hyperbolic  provided that the errors in the pseudo-orbit are sufficiently small. By integrating and extending the methodologies established in \cite{LC,PD3,HZ,ZF}, we prove that these pseudo-orbits possess the bi-shadowing property and periodic bi-Shadowing. 

The main results of this paper are Theorem \ref{th:2.1}, Theorem \ref{th:3.1}, Theorem \ref{th:4.1} and Theorem \ref{th:4.2}.
The paper is organized as follows. In Section \ref{Section 2}, we give the concept of quasi-semi hyperbolic pseudoorbits and an important Theorem \ref{th:2.1}. The results is that quasi-semi hyperbolicity implies quasi hyperbolicity under suitable conditions. In Section \ref{Section 3},  our main result is that quasi-semi-hyperbolic pseudoorbit have the finite bi-shadowing property. In Section \ref{Section 4}, we demonstrate that infinite quasi-semi-hyperbolic pseudoorbit have the finite bi-shadowing property and periodicity.

\section{Quasi-semi hyperbolicity implies quasi hyperbolicity}\label{Section 2}

Everywhere in this paper, we assume that $ M $ is a smooth m-dimensional compact Riemannian
manifold without boundary and denote by Diff($ M $) the set of all $ C^1 $ diffeomorphisms on $ M $. We denote by $ |\cdot| $ and $ \rho (\cdot, \cdot) $ the norm on $ TM $ and the distance on $ M $ induced by the Riemannian metric respectively.

Since $ M $ is compact, we can take a constant $ \varepsilon_0 > 0 $  such that for any $ x\in M$, the standard exponential mapping $\exp_x:\{v\in T_{x}M:|v|<\varepsilon_0\} \rightarrow M $ is a $ C^\infty $ diffeomorphism to its image. Clearly, $ \rho (x, \exp_x(v))=|v| $ for $ v\in T_{x}M $ with $ |v|<\varepsilon_0 $.

Now we give the definition of quasi-semi-hyperbolic orbit segment and quasi-semi-hyperbolic pseudo-orbit.

\begin{definition}\label{defn:2.1}
	Let $f \in\operatorname{Diff}(M)$. We say an orbit arc $ \{x,n\}=\{x,fx,f^2x,\cdots,f^nx\} $ is quasi-semi hyperbolic if following conditions are satisfied:\\
	(1) there are splittings $T_{f^jx}M=E^u_{f^jx}\oplus E^s_{f^jx}, 0\leqslant j\leqslant n$, and $D_{f^jx}f$ is represented by
	$$ \left(\begin{array}{lll}
		A_{j}&B_{j}\\
		C_{j}&D_{j}
	\end{array}\right),\ 0\leqslant j\leqslant n-1;$$
	(2) there exist positive numbers $ \varepsilon $ and $\lambda\in (0,1) $ such that
	\begin{equation}\label{2.1}
		\begin{aligned}
			&\quad \prod_{j=0}^{k-1} \left\|D_{j}\right\|\leqslant \lambda^{k} &&\quad  k=1,2,\cdots,n,\\
			&\quad \prod_{j=k}^{n-1} m(A_{j})\geqslant \lambda^{k-n} &&\quad  k=0,1,\cdots,n-1,\\
			&\quad \frac{\left\|D_{j}\right\|}{m(A_{j})}\leqslant\lambda^{2}  &&\quad j=0,1,\cdots,n-1,\\
			&\quad \left\|B_{j}\right\|,\left\|C_{j}\right\|\leqslant \varepsilon && \quad j=0,1,\cdots,n-1. \\
		\end{aligned}
	\end{equation}
\end{definition}
here $ m(\cdot) $ is the minimum norm.

In above definition, we also say that $\{x,n\}$ is $(\lambda,\varepsilon)$-quasi-semi hyperbolic with respect to the splitting $T_{\boldsymbol f^jx}M=E^u_{f^jx}\oplus E^s_{f^jx}, 0\leqslant j\leqslant n$.

\begin{definition}\label{defn:2.2}
	\hbadness=10000
	\hfuzz=5pt 
	\overfullrule=0pt
	Let $f \in \operatorname{Diff}(M)$. An orbit segment  $\{x ,n\}=\{x,fx,f^2x,\cdots,f^nx\}$ is quasi-hyperbolic, if it is quasi-semi hyperbolic and the corresponding  splitting 
	$T_{ f^jx}M=E^u_{ f^jx}\oplus E^s_{ f^jx}, 0\leqslant j\leqslant n$ is $Df$-invariant, i.e., $B_{j},C_{j}=0$ for all $0\leqslant j\leqslant n-1$.
\end{definition}

Let $\{x_{i}\}_{i=-\infty}^{+\infty}$ be a (two-sided) sequence points in $ M $ and $\{n_{i}\}_{i=-\infty}^{+\infty}$ be a (two-sided) sequence of positive integers. Denote by $$ \{x_{i},n_{i}\}_{i=-\infty}^{+\infty}=\{f^{j}x_{i}:0\leqslant j\leqslant n_{i}-1\}.$$

\begin{definition}\label{defn:2.3}
	Let $f \in \operatorname{Diff}(M),\lambda,\varepsilon\in (0,1)$. 
	$ \{x_{i},n_{i}\} $ is called a $ (\lambda,\varepsilon,\delta) $-quasi-semi hyperbolic pseudo orbit with respect to $T_{x_{i}} M=E_{x_{i}}^{u}  \oplus E_{x_{i}}^{s}$ if for any $ i $, $ \{x_{i},n_{i}\} $ is $ (\lambda,\varepsilon) $-quasi-semi hyperbolic with respect to $T_{x_{i}} M=E_{x_{i}}^{u}  \oplus E_{x_{i}}^{s}$ and $ \rho(f^{n_{i}}x_{i},x_{i+1})\leqslant \delta $.
	Moreover, $ \{x_{i},n_{i}\} $ is $ (\lambda,\delta) $-quasi hyperbolic pseudo-orbit if $\varepsilon=0$.
\end{definition}
If we denote  $  y_{j}=f^{j-N_{i}}x_{i}$ for $ N_{i}\leqslant j\leqslant N_{i+1}-1$, from the definition above, we see that $y_{j+1}=f_{j}(y_{j})$ if $j=N_{i}, \cdots, N_{i+1}-2$,
where $ N_{i} $ is defined as
$$
N_{i}=\left\{\begin{array}{ll}
	0, & \text { if } i=0 \\
	n_{0} +n_{1}+\cdots+n_{i-1}, & \text { if } i>0\\
	n_{i}+n_{i+1}+\cdots+n_{-1}, & \text { if } i<0
\end{array}\right.
$$

For a subset $\Lambda\subset M$ and a continuous splitting $T_{\Lambda}M=E^u \oplus E^s$, a norm $|\cdot |_B$ on $ T_{\Lambda}M $  is said to be of box type​ (with respect to the splitting) if
$$ |v |_B=\max\{|v^u|, |v^s| \}, \ \forall v\in T_{\Lambda}M,$$
where $|\cdot|$ is the norm reduced by the Riemannian metric.
\begin{remark}\label{rem:2.1}
	If a finite orbit segment $\{x ,n\}=\{x,fx,f^2x,\cdots,f^nx\}$ is quasi-semi-hyperbolic with respect to the original Riemannian norm, then it remains quasi-semi-hyperbolic with respect to any equivalent box-type norm(see the proof of Theorem 2.1 in \cite{LC}).
\end{remark}

Now we give a main result which establishes that a quasi-semi-hyperbolic pseudo-orbit is actually quasi-hyperbolic with sufficiently small errors.

\begin{theorem}\label{th:2.1}
	Let $f \in \operatorname{Diff}(M),0 < \lambda < \widetilde{\lambda} < 1 $, and let $TM=E^u \oplus E^s$ be a continuous splitting. Then there exist two positive numbers $\varepsilon_0$ and $\delta_0$ such that for any positive number $\varepsilon\leqslant\varepsilon_0 $ and $\delta\leqslant\delta_0$, every $(\lambda,\varepsilon,\delta)$-quasi-semi hyperbolic pseudo-orbit $ \{x_{i},n_{i}\}_{i=-\infty}^{\infty} $ is actually  $(\widetilde{\lambda},\delta)$-quasi hyperbolic with respect to some splitting.
\end{theorem}
\begin{proof}{\em Step 1.} We first construct two numbers $\varepsilon_1,\delta_1$ and give some notations. 
	
	Denote by $ R=\sup_{x\in M}\{\left\|Df_{x}\right\|,\left\|Df_{x}^{-1}\right\|\} $ and choose $\lambda_{0}\in(\lambda,\widetilde{\lambda}), $ and set 
	$$
	\varepsilon_1=\min\{\frac{1-\lambda_{0}^{2}}{(\lambda_{0}^{2}+6)R},\frac{\lambda_{0}-\lambda}{\lambda_{0}R}\}.
	$$
	
	For $f \in \operatorname{Diff}(M)$, there exists $\delta_1 > 0$ such that if $x_i,x_{i+1} \in M$ satisfying $\rho(f(x_i), x_{i+1}) \leqslant \delta_1$, then\begin{equation}\label{2.2}|D_0(\exp_{x_{i+1}}^{-1} \circ f \circ \exp_{x_i}) - D_{x_i}f| \leqslant \varepsilon_1.\end{equation}
	
	Suppose that $ \{x_{i},n_{i}\}_{i=-\infty}^{\infty} $ is a $ (\lambda,\varepsilon,\delta) $-quasi-semi hyperbolic pseudo-orbit with $ \varepsilon\leqslant\varepsilon_1 $ and $ \delta\leqslant\delta_1 $. Denote the points $  y_{j}=f^{j-N_{i}}x_{i}$ for $ N_{i}\leqslant j\leqslant N_{i+1}-1 .$ With respect to the splitting $ T_{y_{j}}M=E^{u} \oplus E^{s} $, write the derivative matrices as
	$$
	L_{j}=D_{y_{j}}f=
	\begin{pmatrix}
		A_{j}&B_{j}\\
		C_{j}&D_{j}
	\end{pmatrix}
	$$ 
	，
	$$
	\widetilde{L}_{j}=D_{0}(\text{exp}_{y_{j+1}}^{-1}\circ f \circ \text{exp}_{y_{j}})=
	\begin{pmatrix}
		\widetilde{A}_{j}&\widetilde{B}_{j}\\
		\widetilde{C}_{j}&\widetilde{D}_{j}
	\end{pmatrix},\quad j\in \mathbb Z.
	$$
	
	{\em Step 2.} In this step, we construct the $\widetilde{L}$-invariant subspaces $G^u $ and $G^s$. Without loss of generality, we assume throughout this step that the norm is of the  box type.
	
	First, we construct the unstable subspace $G^u$, which satisfies the following expansion condition:
	$$
	\ m(\widetilde{L}_{j}|_{G_{j}^{u}})\geqslant m(A_{j})- 3\varepsilon_1,\quad\forall j\in \mathbb Z.
	$$
	
	Denote by $ L(E^{u},E^{s}) $ the Banach space of linear maps from $ E^{u} $ to $ E^{s} $ and by $ L(E^{u},E^{s})(1) $ the closed unit ball about the origin. 
	
	For  $ P\in\ L(E^{u},E^{s}) $ and  $ j\in \mathbb Z $,
	$$
	\begin{pmatrix}
		\widetilde{A}_{j}&\widetilde{B}_{j}\\
		\widetilde{C}_{j}&\widetilde{D}_{j}
	\end{pmatrix}
	\begin{pmatrix}
		v\\
		P_{j}v
	\end{pmatrix}
	=
	\begin{pmatrix}
		\widetilde{A}_{j}v+\widetilde{B}_{j} P_{j}v\\
		\widetilde{C}_{j}v+\widetilde{D}_{j} P_{j}v
	\end{pmatrix}
	$$
	hence
	$$
	\widetilde{L}_{j}(\text{gr}(P_{j}))\subset \text{gr}(P_{j+1})
	$$
	if and only if
	$$
	P_{j+1}(\widetilde{A}_{j}v+\widetilde{B}_{j} P_{j}v)=\widetilde{C}_{j}v+\widetilde{D}_{j} P_{j}v.
	$$
	
	Since $\{x_{i}, n_{i}\}_{i=-\infty}^{\infty}$ is $(\lambda, \varepsilon)$-quasi-semi hyperbolic, by (\ref{2.1}) we have $ \left\|B_{j}\right\|,\left\|C_{j}\right\|\leqslant \varepsilon\leqslant\varepsilon_1, $. Then by (\ref{2.2}) We immediately obtain
	\begin{equation}\label{2.3}\left\|\widetilde{B}_{j}\right\|,\left\|\widetilde{C}_{j}\right\| \leqslant 2\varepsilon_1, \quad  j \in \mathbb Z.
	\end{equation}
	then
	$$
	m(\widetilde{A}_{j})\geqslant \lambda_{0}^{-2}\left\|\widetilde{D}_{j}\right\|>2\varepsilon_{1}\geqslant\left\|\widetilde{B}_{j}\right\| \geqslant\ \left\|\widetilde{B}_{j}P_{j}\right\|,
	$$
	hence $$  \widetilde{A}_{j}+\widetilde{B}_{j}P_{j}:E^{u}\rightarrow E^{s} $$ is invertible by the Lipschitz inverse function Theorem(Theorem 2.7 of \cite{LW}).
	
	Define a map:
	$$
	T:L(E^{u},E^{s})(1) \rightarrow L(E^{u},E^{s}),
	$$
	$$
	((T_{P}))_{j+1}=(\widetilde{C}_{j}+\widetilde{D}_{j}P_{j})(\widetilde{A}_{j}+\widetilde{B}_{j}P_{j})^{-1}.
	$$
	
	The problem of finding a linear map $ P $ such that $ \widetilde{L}(\text{gr}(P)) \subset \text{gr}(P) $ is equivalent to finding a fixed point of the operator  $ T $.
	
	By the choosing $ \varepsilon_1\leqslant \frac{(1-\lambda_{0}^{2})}{(\lambda_{0}^{2}+6)R} $ and (\ref{2.2}),(\ref{2.3}), then for any  $ P\in\ L(E^{u},E^{s})(1) $,
	\begin{equation}
		\begin{aligned}
			\left\|(T(P))_{j+1}\right\|&
			\leqslant||\widetilde{C}_{j}+\widetilde{D}_{j}P_{j}||\cdot||(\widetilde{A}_{j}+\widetilde{B}_{j}P_{j})^{-1}|| \\
			&\leqslant\frac{\left\|\widetilde{D}_{j}\right\|+2\varepsilon_1}{m(\widetilde{A}_{j})-2\varepsilon_1} \\
			&\leqslant\frac{\lambda_{0}^{2}(m(A_{j})+\varepsilon_1)+2\varepsilon_{1}}{m(A_{j})-3\varepsilon_1}, \\ 
			&<1, \quad \forall j\in \mathbb Z.\\ 
		\end{aligned}
		\nonumber
	\end{equation}
	
	Hence $ T $ maps $  L(E^{u},E^{s})(1) $ into itself.
	
	Moreover, for any  $ P,P^{'}\in L(E^{u},E^{s})(1) $, we have
	$$
	(T(P))_{j+1}(\widetilde{A}_{j}+\widetilde{B}_{j}P_{j})=\widetilde{C}_{j}+\widetilde{D}_{j}P_{j},
	$$
	$$
	(T(P^{'}))_{j+1}(\widetilde{A}_{j}+\widetilde{B}_{j}P_{j}^{'})=\widetilde{C}_{j}+\widetilde{D}_{j}P_{j}^{'}.
	$$
	Thus
	\begin{align*}
		&((T(P))_{j+1}-(T(P'))_{j+1})\widetilde{A}_{j}
		+(T(P))_{j+1}\widetilde{B}_{j}P_{j}
		-(T(P'))_{j+1}\widetilde{B}_{j}P_{j} \\
		&+(T(P'))_{j+1}\widetilde{B}_{j}P_{j}
		-(T(P'))_{j+1}\widetilde{B}_{j}P_{j'}
		= \widetilde{D}_{j}(P_{j}-P_{j'}),
	\end{align*}
	$$
	((T(P))_{j+1}-(T(P^{'}))_{j+1})(\widetilde{A}_{j}+\widetilde{B}_{j}P_{j})=(\widetilde{D}_{j}-(T(P^{'}))_{j+1}\widetilde{B}_{j})(P_{j}-P_{j}^{'}),
	$$
	$$
	(T(P))_{j+1}-(T(P^{'}))_{j+1}=(\widetilde{D}_{j}-(T(P^{'}))_{j+1}\widetilde{B}_{j})(P_{j}-P_{j}^{'})(\widetilde{A}_{j}+\widetilde{B}_{j}P_{j})^{-1}.
	$$
	Therefore,
	\begin{equation}
		\begin{aligned}
			\left\|(T(P))_{j+1}-(T(P^{'}))_{j+1}\right\|&
			\leqslant||\widetilde{D}_{j}-(T(P^{'}))_{j+1}\widetilde{B}_{j}||\cdot||(\widetilde{A}_{j}+\widetilde{B}_{j}P_{j})^{-1}||\cdot\left\|P_{j}-P_{j}^{'}\right\|\\
			&\leqslant\frac{\lambda_{0}^{2}(m(A_{j})+\varepsilon_1)+2\varepsilon_1}{m(A_{j})-3\varepsilon_1}\left\|P_{j}-P_{j}^{'}\right\| \\
			&<\left\|P_{j}-P_{j}^{'}\right\|, \quad\forall j\in \mathbb Z\\
		\end{aligned}
		\nonumber
	\end{equation}
	
	Thus $ T $ is a contraction on the complete metric space $ L(E^{u},E^{s})(1) $. By the contraction mapping principle, $ T $ has a unique fixed point $ P\in L(E^{u},E^{s})(1) $  such that
	$$
	\widetilde{L}_{j}(\text{gr}(P_{j}))\subset \text{gr}(P_{j+1}), \quad \forall j\in \mathbb Z.
	$$
	In global notation, this invariance property is written as, $ \widetilde{L}(\text{gr}(P))\subset \text{gr}(P) $. Since each map $ \widetilde{L}:T_{x}M \rightarrow T_{x}M $ is a linear isomorphism,
	the inclusion is actually an equality
	$$
	\widetilde{L}(\text{gr}(P))= \text{gr}(P).
	$$
	Then we define the unstable subspace by $G_j^u = \operatorname{gr}(P_j)$, then $ \{G_j^{u}\}_{j\in \mathbb Z} $ is invariant under $ \widetilde{L}_{j} $.
	
	Recall that the norm on each tangent space is taken to be of box type since  $ P\in\ L(E^{u},E^{s})(1) $, the norm of a vector in $ \textbf{gr}(P) $ is simply the norm of its first component. Consequently, for any $ v\in G_{j}^{u}  ,$
	$$
	\left\|\widetilde{L}_{j}(v,P_{j}v)\right\|=\left\|\widetilde{A}_{j}v+\widetilde{B}_{j}P_{j}v\right\|\geqslant (m(\widetilde{A}_{j})- 2\varepsilon_1 ) \left\|v\right\|\geqslant (m(A_{j})- 3\varepsilon_1 ) \left\|v\right\|.
	$$
	This establishes that the restriction of $ \widetilde{L}_{j} $ to the unstable subspace $G_{j}^u$ satisfies the required expansion condition:
	$$
	m(\widetilde{L}_{j}|_{G_{j}^{u}})\geqslant m(A_{j})- 3\varepsilon_1, \quad \forall j \in \mathbb{Z}.
	$$
	
	Similarly, by the same graph‑transform， we can get invariant family of graphs $ Q=\{Q_j\} $ with $ Q_j\in\ L(E^{s},E^{u})(1) $. Define the stable subspace by $ G_j^s = \operatorname{gr}(Q_j)$, then we construct the $\widetilde{L}$-invariant subspace $G^s$ which satisfy the following estimates:$$ \|\widetilde{L}_j|_{G^s_j}\| \leqslant \|D_{j}\| + 3\varepsilon_1, \quad \forall j \in \mathbb{Z} .$$ Thus we have produced a $\widetilde{L}$-invariant decomposition $T     M = G^u \oplus G^s$. With respect to this splitting the pulled‑back derivative takes a block‑diagonal form:
	$$
	F_{j}=\begin{pmatrix}
		M_{j}&0\\
		0&N_{j}
	\end{pmatrix},\quad j\in\mathbb{Z}.
	$$ 
	
	{\em Step 3.} Complete the proof of
	Theorem \ref{th:2.1}. 
	
	Since the construction of the invariant decomposition relies on the fixed‑point method. This guarantees continuous dependence on the parameters, from which it follows that there exist two numbers, $\varepsilon_0 \leqslant \varepsilon_1$ and $\delta _0 \leqslant \delta_1$, such that the new splitting $G^u \oplus G^s$ is sufficiently close to the original splitting $E^u \oplus E^s$. Hence, for $j \in \mathbb{Z}$, we have:
	$$
	\left\|M_{j}-\widetilde{A}_{j}\right\|\leqslant \frac{\widetilde{\lambda}-\lambda_{0}}{\widetilde{\lambda}R}.
	$$ 
	$$
	\left\|N_{j}-\widetilde{D}_{j}\right\|\leqslant \frac{\widetilde{\lambda}-\lambda_{0}}{\widetilde{\lambda}R}.
	$$ 
	Then we obtain
	\begin{equation}\label{2.4}\frac{\lambda_{0}}{\widetilde{\lambda}}m(\widetilde{A}_{j})\leqslant m(M_{j}) \leqslant \left\|M_{j}\right\|\leqslant \frac{\widetilde{\lambda}}{\lambda_{0}}\left\|\widetilde{A}_{j}\right\|.
	\end{equation}
	\begin{equation}\label{2.5}\frac{\lambda_{0}}{\widetilde{\lambda}}m(\widetilde{D}_{j})\leqslant m(N_{j}) \leqslant \left\|N_{j}\right\|\leqslant \frac{\widetilde{\lambda}}{\lambda_{0}}\left\|\widetilde{D}_{j}\right\|.
	\end{equation}
	Since $ \varepsilon_1\leqslant\frac{\lambda_{0}-\lambda}{\lambda_{0}R}\leqslant\frac{\lambda_{0}-\lambda}{\lambda R} $, one has 
	$$
	\left\|\widetilde{D}_{j}\right\|\leqslant \left\|D_{j}\right\|+\varepsilon_1\leqslant \frac{\lambda_{0}}{\lambda}\left\|D_{j}\right\|,j \in \mathbb Z
	$$
	$$
	m(\widetilde{A}_{j})\geqslant m(A_{j})-\varepsilon_1\geqslant\frac{\lambda}{\lambda_{0}}m(A_{j}),j \in \mathbb Z.
	$$ 
	Then, for $j=N_{i},N_{i}+1,\cdots,N_{i+1}-1$, 
	$$
	\frac{\left\|\widetilde{D}_{j}\right\|}{m(\widetilde{A}_{j})}\leqslant\frac{\lambda_{0}^{2}\left\|D_{j}\right\|}{\lambda^{2}m(A_{j})}\leqslant \frac{\lambda_{0}^{2}}{\lambda^{2}} \cdot \lambda^{2} =\lambda_{0}^{2},
	$$ 
	By (\ref{2.4}) and (\ref{2.5}) , we have 
	\begin{equation}\label{2.6}
		\frac{\left\|N_{j}\right\|}{m(M_{j})}\leqslant \frac{\frac{\widetilde{\lambda}}{\lambda_{0}}\left\|\widetilde{D}_{j}\right\|}{\frac{\lambda_{0}}{\widetilde{\lambda}}m(\widetilde{A}_{j})} \leqslant \frac{\widetilde{\lambda}^{2}}{\lambda_{0}^{2}}\cdot \lambda_{0}^{2}= \widetilde{\lambda}^{2}.
	\end{equation}
	
	For $h=N_{i}+1,N_{i}+2,\cdots,N_{i+1}$,
	$$
	\prod_{j=N_{i}}^{h-1} \left\|\widetilde{D}_{j}\right\|\leqslant \prod_{j=N_{i}}^{h-1} \frac{\lambda_{0}}{\lambda}\left\|D_{j}\right\| \leqslant (\frac{\lambda_{0}}{\lambda})^{h-N_{i}}\cdot\lambda^{h-N_{i}}=\lambda_{0}^{h-N_{i}}.
	$$ 
	Then, by (\ref{2.4}), we have 
	\begin{equation}\label{2.7}
		\prod_{j=N_{i}}^{h-1} \left\|N_{j}\right\|\leqslant \prod_{j=N_{i}}^{h-1} \frac{\widetilde{\lambda}}{\lambda_{0}}\left\|\widetilde{D}_{j}\right\| \leqslant (\frac{\widetilde{\lambda}}{\lambda_{0}})^{h-N_{i}}\cdot\lambda_{0}^{h-N_{i}}=\widetilde{\lambda}^{h-N_{i}}.
	\end{equation}
	
	For $h=N_{i},N_{i}+1,\cdots,N_{i+1}-1$,
	$$
	\prod_{j=h}^{N_{i+1}-1} m(\widetilde{A}_{j})\geqslant \prod_{j=h}^{N_{i+1}-1} \frac{\lambda}{\lambda_{0}}m(A_{j})\geqslant(\frac{\lambda}{\lambda_{0}})^{N_{i+1}-h}\cdot\lambda^{h-N_{i+1}}=\lambda_{0}^{h-N_{i+1}}.
	$$ 
	Then, by (\ref{2.4}), we have 
	\begin{equation}\label{2.8}
		\prod_{j=h}^{N_{i+1}-1} m(M_{j})\geqslant \prod_{j=h}^{N_{i+1}-1} \frac{\lambda_{0}}{\widetilde{\lambda}}m(\widetilde{A}_{j}) \geqslant (\frac{\lambda_{0}}{\widetilde{\lambda}})^{N_{i+1}-h}\cdot\lambda_{0}^{h-N_{i+1}}=\widetilde{\lambda}^{h-N_{i+1}}.
	\end{equation}
	Therefore, from (\ref{2.6}),(\ref{2.7}),(\ref{2.8}), we conclude that $ \{x_{i},n_{i}\}_{i=-\infty}^{\infty} $ is actually  $(\widetilde{\lambda},\delta)$-quasi hyperbolic with respect to splittings $T_{f^{j-N_{i}}x}M=E^u_{f^{j-N_{i}}x}\oplus E^s_{f^{j-N_{i}}x} $ for $ N_{i}\leqslant j\leqslant N_{i+1}-1 $ and satisfies
	\begin{equation*}
		\begin{alignedat}{2}
			&\quad \prod_{j=N_{i}}^{h-1} \left\|N_{j}\right\|\leqslant \widetilde{\lambda}^{h-N_{i}}, && \quad h=N_{i}+1,N_{i}+2,\cdots,N_{i+1},  \\
			&\quad \prod_{j=h}^{N_{i+1}-1} m(M_{j})\geqslant \widetilde{\lambda}^{h-N_{i+1}}, &&\quad h=N_{i},N_{i}+1,\cdots,N_{i+1}-1, \\
			&\quad \frac{\left\|N_{j}\right\|}{m(M_{j})}\leqslant\widetilde{\lambda}^{2}, &&\quad j=N_{i},N_{i}+1,\cdots,N_{i+1}-1,\\ 
		\end{alignedat}
	\end{equation*}
	This completes the proof.
\end{proof}

\section{Bi-shadowing of Finite Quasi-semi-hyperbolic pseudo-orbit}\label{Section 3}
\hbadness=10000
\hfuzz=15pt 
\overfullrule=0pt 
In this section, we will show that finite quasi-semi-hyperbolic pseudo-orbit has the finite bi-shadowing property. First, we recall some concepts and lemmas of \cite{SG} which will be used to prove Theorem \ref{th:3.1}.

\begin{definition}

Let $ \lambda\in (0,1) $. A pair of sequences $\{a_{i},b_{i}\}_{i=1}^{n}  $ of positive numbers is called $ \lambda$-hyperbolic if:
$$
a_{k}\leqslant\lambda,\quad b_{k}\geqslant \lambda^{-1} \quad \  k=1,2,...,n.
$$
A pair of sequences $\{a_{i},b_{i}\}_{i=1}^{n}  $ of positive numbers is called $ \lambda $-quasi-hyperbolic if the following three conditions are satisfied:\\
(1) $ \quad \prod_{j=1}^{k} a_{j}\leqslant \lambda^{k}  $,\\
(2) $ \quad \prod_{j=k}^{n} b_{j}\geqslant \lambda^{k-n-1}  $,\\
(3) $ \quad \frac{b_{k}}{a_{k}}\leqslant\lambda^{2}   $,\\                                                                                                                                                                                                                                                                                                                                                                                                                                                                                                                                                                                                                                                                                                                                                                                                                                                                                                                                                                                                                                                                                                                                                                                                                                                                                                                                                                                                                                                                                                                                                                                                                                                                                                                                                                                                                                                                                                                                                                                                                                                                                                                                                                                                                                                                                                                                                                                                                                                                                                                                                                                                                                                                                                                                                                                                                                                                                                                                                                                                                                                                                                                                                                                                                                                                                                                                                                                                                                                                                                                                                                                                                                                                                                                                                                                                                                                                                                                                                                                                                                                                                                                                                                                                                                                                                                                                                                                                                                                                                                                                                                                                                                                                                                                                                                                                                                                                                                                                                                                                                                                                                                                                                                                                                                                                                                                                                                                                                                                                                                                                                                                                                                                                                                                                                                                                                                                                                                                                                                                                                                                                                                                                                                                                                                                                                                                                                                                                                                                                                                                                                                                                                                                                                                                                                                                                                                                                       
for $\ k=1,2,...,n.  $
\end{definition}

\begin{definition}
	A sequence $\{c_i\}_{i=1}^n$ of positive numbers is
	called a balance sequence if
	$$
	\quad \prod_{j=1}^k c_j\leqslant 1, k=1,2,...,n-1,\quad
	\prod_{j=1}^n c_j=1.
	$$
	
	A balance sequence $\{c_i\}_{i=1}^n$ is called adapted to
	$\lambda$-quasi-hyperbolic sequence $\{a_i,b_i\}_{i=1}^n$
	if
	\begin{align*}
		\left\{\frac{a_i}{c_i},\frac{b_i}{c_i}\right\}_{i=1}^n
	\end{align*}
	is still $\lambda$-quasi-hyperbolic. Moreover, if
	\begin{align*}
		\left\{\frac{a_i}{c_i},\frac{b_i}{c_i}\right\}_{i=1}^n
	\end{align*}
	is $\lambda$-hyperbolic, then $\{c_i\}_{i=1}^n$ is called
	well-adapted.
\end{definition}
\begin{lemma}\label{lem:3.1}
Given $ \lambda\in (0,1) $, any $ \lambda $-quasi-hyperbolic pair of sequences $ \{a_{i},b_{i}\}_{i=1}^{n}  $ has a well-adapted sequence $\{c_{i}\}_{i=1}^{n} $.
\end{lemma}

\begin{remark}\label{rem:3.1}
If $\{c_{i}\}_{i=1}^{n} $ is a well-adapted sequence of $ \{a_{i},b_{i}\}_{i=1}^{n} $, then we have $ \frac{a_{i}}{c_{i}}\leqslant\lambda $ and $ \frac{b_{i}}{c_{i}}\geqslant\lambda^{-1} $. Hence $ a_{i}<\frac{a_{i}}{\lambda}\leqslant c_{i}\leqslant b_{i}\lambda<b_{i},i=1,2...,n. $
\end{remark}

\begin{definition}\relpenalty=0
	\binoppenalty=0
	\setlength{\emergencystretch}{5em}
	\hbadness=10000
	\hfuzz=20pt
	\overfullrule=0pt
	A sequence  $\{c_{i}\}_{i=1}^{n}  $ of positive numbers is called  a balance sequence if
	\begin{equation*}
	\quad \prod_{j=1}^{k} c_{j}\leqslant 1, k=1,2,...,n-1,\quad \prod_{j=1}^{n} c_{j}=1.	\end{equation*}
	A balance sequence $\{c_{i}\}_{i=1}^{n}  $ is called adapted to a $ \lambda$-quasi-hyperbolic sequence $ \{a_{i},b_{i}\}_{i=1}^{n}  $  if $\{\frac{a_{i}}{c_{i}},\frac{b_{i}}{c_{i}}\}_{i=1}^{n}  $ is still $ \lambda $-quasi-hyperbolic. Moreover, if $\{\frac{a_{i}}{c_{i}},\frac{b_{i}}{c_{i}}\}_{i=1}^{n}  $ is $ \lambda $-hyperbolic, then $\{c_{i}\}_{i=1}^{n}  $ is called well-adapted.
\end{definition}

\begin{theorem}\label{th:3.1}
Let $f \in \operatorname{Diff}(M),0 < \lambda < 1 $, and let $TM=E^u \oplus E^s$ be a continuous splitting. For given $ k,\varepsilon_1>0 $ Then there exist three positive numbers   $\delta_0, \varepsilon_0$ and $d_0$ with the following property: For any $g \in \operatorname{Diff}(M)$ satisfying $ \|f - g\|_\infty\leqslant d \leqslant d_0 $​, and for any  $(\lambda,\varepsilon,\delta)$-quasi-semi hyperbolic pseudo-orbit $  \{x_{i},n_{i}\}_{i=-k}^{k} $ of $ f $ with $ \delta\leqslant\delta_0,\varepsilon\leqslant\varepsilon_0 $, there exists a point $ x\in M $ that $\varepsilon_1 $-bi-shadows $ \{x_{i},n_{i}\}_{i=-k}^{k}, $ i.e., $ \rho(g^j x,f^{j-N_{i}}x_{i})\leqslant\varepsilon_1 $ for $ N_{i}\leqslant j\leqslant N_{i+1}-1 $.
\end{theorem}
\vspace{1em}

\begin{proof}
	{\em Step 1.} We begin by constructing the three positive numbers  $\delta_0, \varepsilon_0$ and $d_0$ specified in the theorem and giving some notations.
	
	For a given $ i\in\lbrack -k,k \rbrack $, define the points $$  y_{j}=f^{j-N_{i}}x_{i}\quad \text{for}\quad N_{i}\leqslant j\leqslant N_{i+1}-1 .$$
	Denote by $$ X_{j}=T_{y_{j}}M,E_{j}^{u}=E_{y_{j}}^{u},E_{j}^{s}=E_{y_{j}}^{s},R=\sup_{x\in M}\{\left\|Df_{x}\right\|,\left\|Df_{x}^{-1}\right\|\}. $$
	
	First, choose a positive constant $\widetilde{\lambda} \in (\lambda, \frac{1+\lambda}{2})$. Define $$\varepsilon_0 = \frac{1 + \lambda - 2\widetilde{\lambda}}{4R}. $$

	For $f \in \operatorname{Diff}(M)$, there exists $\eta\in(0,\varepsilon_1) $ such that for every $ y_j \in M $ satisfying:
	$$ |D_\xi(\exp_{f(y_{j})}^{-1} \circ f \circ \exp_{y_j}) - D_{y_j}f| \leqslant \frac{\widetilde{\lambda} - \lambda}{5R}.\quad \forall \xi \in X_{j}(\eta). $$
	Furthermore, for each $j \in \mathbb{Z}$, the uniform continuity guarantees the existence of a number $\delta_{1} \in (0, \frac{1-\widetilde{\lambda}}{4}\eta)$ such that if $y_j,y_{j+1} \in M$  satisfy $d(f(y_j), y_{j+1}) \leqslant \delta_{1}$, then:
	\begin{equation}\label{3.1}\left| D_\xi (\exp_{f(y_{j})}^{-1} \circ f \circ \exp_{y_j}) - D_0 (\exp_{y_{j+1}}^{-1} \circ f\circ \exp_{y_j}) \right| \leqslant \frac{\widetilde{\lambda} - \lambda}{4R}, \quad \forall \xi \in X_{j}(\eta),
	\end{equation}
	and
	\begin{equation}\label{3.2}\left| D_0 (\exp_{y_{j+1}}^{-1} \circ f \circ \exp_{y_j}) - D_{y_j} f\right| \leqslant \frac{\widetilde{\lambda} - \lambda}{4R}.
	\end{equation}
	
	For fixed $ k $, let $$ a=\max_{i\in [-k,k]}\{n_{i}\} $$ and then denote $ C=R^{a} $. Finally,  set  $$\delta_0 = \frac{\delta_{1} }{C},\quad d_0 = \frac{1-\widetilde{\lambda}}{4C}\eta .$$
	
	Let $ 0<\delta\leqslant\delta_0,0<\varepsilon\leqslant\varepsilon_0  $ and consider $ (\lambda,\varepsilon,\delta) $-quasi-semi hyperbolic pseudoorbit $ \{x_{i},n_{i}\}_{i=-k}^{k} $. Without loss of generality, we may assume that $ i\in [0,k] $. Relative to the splitting $ TM=E^{u} \oplus E^{s} $, the derivative of the local representation of $f$ at the points $  y_{j} $ is given by the matrices
	$$
	L_{j}=D_{0}(\text{exp}_{f(y_{j})}^{-1}\circ f\circ \text{exp}_{y_{j}})=
	\begin{pmatrix}
		A_{j}&B_{j}\\
		C_{j}&D_{j}
	\end{pmatrix},\quad j=0,1,\cdots,N_{k+1}-1.
	$$ 
	
	Consider any $g \in \operatorname{Diff}(M)$ satisfying $ \rho(f,g)\leqslant d\leqslant d_0 $. For each $j=0,1,\cdots,N_{k+1}-1$ and any $v \in X_{j}$, we define the local representations of $ f $ and $ g $: 
	$$
	F_{j}(v)=\text{exp}_{y_{j+1}}^{-1}\circ f\circ \text{exp}_{y_{j}}(v)  \text{  and  } G_{j}(v)=\text{exp}_{y_{j+1}}^{-1}\circ g\circ \text{exp}_{y_{j}}(v). 
	$$
	{\em Step 2.} In this step, we define a collection of new norms $|\cdot|_{j,N} $ on each tangent space $X_{j}$ and construct an important mapping.
	
	Since for any $ i $, $ \{x_{i},n_{i}\} $ is $ (\lambda,\varepsilon) $-quasi-semi hyperbolic, this means:
	\begin{equation*}
		\begin{alignedat}{2}
			&\quad \prod_{j=N_{i}}^{h-1} \left\|D_{j}\right\|\leqslant \lambda^{h-N_{i}}, &&  \quad h=N_{i}+1,N_{i}+2,\cdots,N_{i+1}-1,  \\
			&\quad \prod_{j=h}^{N_{i+1}-1} m(A_{j})\geqslant \lambda^{h-N_{i+1}},&& \quad h=N_{i},N_{i}+1,\cdots,N_{i+1}-1,\\
			&\quad \frac{\left\|D_{j}\right\|}{m(A_{j})}\leqslant\lambda^{2},  && \quad j=N_{i},N_{i}+1,\cdots,N_{i+1}-1.\\
			&\quad \left\|B_{j}\right\|,\left\|C_{j}\right\|\leqslant \varepsilon && \quad j=N_{i},N_{i}+1,\cdots,N_{i+1}-1.
		\end{alignedat}
	\end{equation*}
	The above equations imply that the pair of sequences $\{||B_{22j}||,m(B_{11j})\}_{j= N_{i}}^{N_{i+1}-1}$ is $ \lambda $-quasi-hyperbolic pair of sequences. By Lemma \ref{lem:3.1} and Remark \ref{rem:3.1}, $\{||B_{22j}||,m(B_{11j})\}_{j= N_{i}}^{N_{i+1}-1}$ has a well-adapted sequence $\{h_{j}\}_{j= N_{i}}^{N_{i+1}-1} $ satisfying
	\begin{equation}\label{3.3}h_{j}^{-1}m(A_{j})>\lambda^{-1}\quad \text{and}\quad h_{j}^{-1}\left\|D_{j}\right\|<\lambda,\quad \forall j=N_{i},N_{i}+1,\cdots,N_{i+1}-1.
	\end{equation}
	\begin{equation}\label{3.4} \frac{1}{R}\leqslant h_{j}\leqslant R,\quad\forall j=N_{i},N_{i}+1,\cdots,N_{i+1}-1.
	\end{equation}
	
	Define $ l_{j}=\prod_{k=N_{i}}^{j-1}h_{k} $ for $ j=N_{i}+1,N_{i}+2,\cdots,N_{i+1}-1 $(denote $ l_{N_{i}-1}=1 $). From the definition of a balance sequence, we know that  $ l_{j}\leqslant 1 $.
	For each $j=N_{i},N_{i}+1,\cdots,N_{i+1}-1$, we introduce a new norm $|\cdot|_N = |\cdot|_{j,N}$ on $X_{j}$ by $$|v|_N = l_j^{-1} |v|, \quad \forall v \in X_{j}.$$
	Since $$m_N(A_{j}) = \inf_{v \in E^u_j, v \neq 0} \frac{|A_{j}v|_N}{|v|_N} = \inf_{v \in E^u_j, v \neq 0} \frac{l_{j+1}^{-1} |A_{j}v|}{l_j^{-1} |v|} = h_j^{-1} m(A_{j}),$$$$\|D_{j}\|_N = \sup_{v \in E^s_j, v \neq 0} \frac{|D_{j}v|_N}{|v|_N} = \sup_{v \in E^s_j, v \neq 0} \frac{l_{j+1}^{-1} |D_{j}v|}{l_j^{-1} |v|} = h_j^{-1} \|D_{j}\|.$$
	By (\ref{3.3}), we have
	\begin{equation}\label{3.5} m_{N}(A_{j}) > \lambda^{-1} \quad \text{and} \quad \|D_{j}\|_N< \lambda.
	\end{equation}
	
	For a given  $ v_{j}\in X_{j} $ with $ |P_{j}^{s}v_{j}|_{N}\leqslant\eta $ define the mapping  $ \Phi_{v_{j}}:E_{j}^{u}(\eta)\rightarrow E_{j+1}^{u} $ on the unstable ball $ E_{j}^{u}(\eta)=\{v\in E_{j}^{u}:|v|_{N}\leqslant\eta\} $ by
	$$
	\Phi_{v_{j}}(w_{j})=P_{j+1}^{u}(F_{j}(P_{j}^{s}v_{j}+w_{j})-F_{j}(P_{j}^{s}v_{j})).
	$$ Here, $ P_{j}^{s}:X_{j} \rightarrow E_{j}^{s}  $ and $ P_{j}^{u}:X_{j} \rightarrow E_{j}^{u} $ are the projection operators defined by:
	$$
	P_{j}^{s}X_{j}=E_{j}^{s},\quad P_{j}^{s}E_{j}^{u}=0;
	$$
	$$
	P_{j}^{u}X_{j}=E_{j}^{u},\quad P_{j}^{u}E_{j}^{s}=0.
	$$
	
	Furthermore, expressing the local representation of $ f $ in components as:
	$$
	\text{exp}_{y_{j+1}}^{-1}\circ f \circ \text{exp}_{y_{j}}=(\alpha_{j}^{u},\alpha_{j}^{s}),\quad\forall j=0,1,\cdots,N_{k+1}-1,
	$$
	then the map $\Phi_{v_j}$ takes the form:
	$$
	\Phi_{v_{j}}(w_{j})=\alpha_{j}^{u}(P_{j}^{s}v_{j}+w_{j})-\alpha_{j}^{u}(P_{j}^{s}v_{j}).
	$$
	
	{\em Step 3.} We present several important lemmas that will be used in the proof of Theorem \ref{th:3.1}.
	
	\begin{lemma}\label{lem:3.2} Given $ i\in [0,k] $. For each $j=N_{i},N_{i}+1,\cdots,N_{i+1}-1.$ the mapping $\Phi_{v_j}$ is expanding. Specifically, it satisfies
		$$|\Phi_{v_j}(w_j) - \Phi_{v_j}(w_j')|_N \geqslant \widetilde{\lambda}^{-1} |w_j - w_j'|_N, \quad \forall w_j, w_j' \in E_{j}^u(\eta).$$ 
	\end{lemma}
	\begin{proof}[Proof of Lemma \ref{lem:3.2}]
		Given $ i\in [0,k] $. For $ j=N_{i},N_{i}+1,\cdots,N_{i+1}-2, $ by the first equality of (\ref{3.3})gives directly
		$$
		m_N\left(\frac{\partial}{\partial u} \alpha_j^u(0)\right)= m_N(A_{j})\geqslant  \lambda^{-1}> (\frac{{\widetilde{\lambda}+\lambda}}{2}) ^{-1}. 
		$$
		For the terminal index $ j=N_{i+1}-1 $, by (\ref{3.2}), (\ref{3.5}) and the property $ l_{j}\leqslant 1 $ imply that
		$$\begin{alignedat}{2}
			&m_N\left(\frac{\partial}{\partial u} \alpha_{N_{i+1}-1}^u(0)\right)& &= h^{-1}_{N_{i+1}-1} m\left(\frac{\partial}{\partial u} \alpha_{N_{i+1}-1}^u(0)\right) \\
			& &&\geqslant h^{-1}_{N_{i+1}-1} \left(m(A_{N_{i+1}-1 })-\left\|\frac{\partial}{\partial u} \alpha_{N_{i+1}-1}^u(0) -A_{N_{i+1}-1 }\right\|\right)\\
			& &&\geqslant   m_N(A_{N_{i+1}-1 })-\frac{\widetilde\lambda-\lambda}{4} \\
			& &&\geqslant    {\lambda}^{-1} -\frac{\widetilde\lambda-\lambda}{4}\\
			& && \geqslant (\frac{{\widetilde{\lambda}+\lambda}}{2}) ^{-1}.
		\end{alignedat}
		$$
		Thus, for every​$ j=N_{i},N_{i}+1,\cdots,N_{i+1}-1, $ we have the uniform lower bound 
		$$
		m_N\left(\frac{\partial}{\partial u} \alpha_j^u(0)\right)> (\frac{{\widetilde{\lambda}+\lambda}}{2}) ^{-1}. 
		$$
		Using inequality (\ref{3.1}), which stems from the uniform continuity of $ Df $, we can control the variation of the derivative. For any $ \xi \in X_{j}(\eta),\  j=N_{i},N_{i}+1,\cdots,N_{i+1}-1,$
		\begin{equation}
			\begin{aligned}
				\left\|\frac{\partial}{\partial u} \alpha_j^u(\xi)- \frac{\partial}{\partial u} \alpha_j^u(0) \right\|_N
				&\leqslant R \left\|\frac{\partial}{\partial u} \alpha_j^u(\xi)- \frac{\partial}{\partial u} \alpha_j^u(0) \right\|\\
				&=R \left\|D_\xi \exp _{y_{j+1}}^{-1} \circ f_j \circ \exp _{y_j} -D_0 \exp _{y_{j+1}}^{-1} \circ f_j \circ \exp _{y_j}  \right\|\\
				&\leqslant R \frac{\widetilde\lambda-\lambda}{4R}\\
				&\leqslant (\frac{{\widetilde{\lambda}+\lambda}}{2}) ^{-1}-\tilde\lambda^{-1}.
			\end{aligned}
			\nonumber
		\end{equation}
		This implies that for all $ \xi \in X_{j}(\eta),$
		\begin{equation}
			\begin{aligned}
				m_N\left(\frac{\partial}{\partial u} \alpha_j^u(\xi)\right)
				&\geqslant m_N\left(\frac{\partial}{\partial u} \alpha_j^u(0)\right)- \left\|\frac{\partial}{\partial u} \alpha_j^u(\xi)- \frac{\partial}{\partial u} \alpha_j^u(0) \right\|_N \\
				&\geqslant (\frac{{\widetilde{\lambda}+\lambda}}{2}) ^{-1}-((\frac{{\widetilde{\lambda}+\lambda}}{2}) ^{-1}-\widetilde\lambda^{-1})= \widetilde\lambda^{-1}.\\
			\end{aligned}
			\nonumber
		\end{equation}
		
		Let $ w_{j},w_{j}^{'}\in E_{j}^{u}(\eta) $. Define  $ \tau(t)=P_{j}^{s}v_{j}+w_{j}+t(w_{j}^{'}-w_{j}) $ for $ t\in[0,1] $. By the mean value theorem and the definition of $ F_{v_{j}} $, we have
		\begin{equation}
			\begin{aligned}
				|\Phi_{v_{j}}(w_{j})-\Phi_{v_{j}}(w_{j}^{'})|_{N}
				&=|\alpha_{j}^{u}(P_{j}^{s}v_{j}+w_{j})-\alpha_{j}^{u}(P_{j}^{s}v_{j}+w_{j}^{'})|_{N} \\
				&=|\int_{0}^{1}\frac{\partial}{\partial u}\alpha_{j}^{u}(\tau(t))(w_{j}-w_{j}^{'})dt|_{N}\\
				&\geqslant \underset{t\in [0,1]}{\inf}m_{N}(\frac{\partial}{\partial u}\alpha_{j}^{u}(\tau(t)))|w_{j}-w_{j}^{'}|_{N}\\
				&\geqslant \widetilde\lambda^{-1}|w_{j}-w_{j}^{'}|_{N}.
			\end{aligned}
			\nonumber
		\end{equation}
		This completes the proof.
	\end{proof}
	The above lemma says that $\Phi_{v_{j}} $ is expanding on $ E_{j}^{u}(\eta) $. Consequently, its inverse map, denoted by
	$$
	\Psi_{v_{j}}:=\Phi_{v_{j}}^{-1}:\Phi_{v_{j}}(E_{j}^{u}(\eta))\rightarrow E_{j}^{u}(\eta)
	$$
	is well defined and is a contraction with Lipschitz constant at most $ \widetilde\lambda .$
	
	We now define the space on which the fixed-point argument will take place. Let
	$$
	\mathcal{B}_{j}^{u}=\{\{v_{j}\}_{j=0}^{N_{k+1}}:v_{j}\in E_{j}^{u}\},  
	$$
	$$
	\mathcal{B}_{j}^{u}(\eta)=\{\{v_{j}\}_{j=0}^{N_{k+1}}\in E_{j}^{u}:|v_{j}|_{N}\leqslant\eta\}.
	$$
	Introduce the operator $ \mathcal{A}:\mathcal{B}_{j}^{u}(\eta) \rightarrow\mathcal{B}_{j}^{u} $ which transforms a sequence $ \boldsymbol{v}=\{v_{j}\}_{j=0}^{N_{k+1}} $ into a new sequence $ \boldsymbol{w}=\{w_{j}\}_{j=0}^{N_{k+1}} $ defined by:
	
	For $ j=0,N_{k+1}-1, $
	\begin{align}\label{3.6}
		P_{0}^{s}w_{0}=0, \quad
		P_{N_{k+1}}^{u}w_{N_{k+1}}=0.
	\end{align}
	
	For $ j=0,\cdot\cdot\cdot,N_{k+1}-1, $
	\begin{align}\label{3.7}
		P_{j+1}^{s}w_{j+1}=P_{j+1}^{s}(G_{j}(v_{j})).
	\end{align}
	
	For $ j=0,\cdot\cdot\cdot,N_{k+1}-1,$
	\begin{align}\label{3.8}
		P_{j}^{u}w_{j}=\Psi_{v_{j}}(P_{j+1}^{u}(-G_{j}(v_{j})+F_{j}(v_{j})-F_{j}(P_{j}^{s}v_{j})+v_{j+1})).
	\end{align}
	
	\begin{lemma}\label{lem:3.3} 	The operator $ \mathcal{A} $ is well defined on $ \mathcal{B}_{j}^{u}(\eta) $.
	\end{lemma}
	\begin{proof}[Proof of Lemma \ref{lem:3.3}]
		For, it is clear by (\ref{3.7}) that the right hand side of $ P_{j+1}^{s}w_{j+1}=P_{j+1}^{s}(G_{j}(v_{j})) $ is well
		defined and depends continuously on $ \boldsymbol{v}\in \mathcal{B}_{j}^{u}(\eta) $. Hence, by Lemma (\ref{lem:3.2}), it is suffices to prove that
		\begin{align}\label{3.9}
			\begin{split}
				|P_{j+1}^{u}(-G_{j}(v_{j})+F_{j}(v_{j})-F_{j}(P_{y_{j}}^{s}v_{j})+v_{j+1})|_{N}\leqslant\widetilde\lambda^{-1}\eta.
			\end{split}
		\end{align}
		
		Rewrite the last inequality in the form
		$$
		|M_{1j}+M_{2j}+M_{3j}|_{N}\leqslant\widetilde\lambda^{-1}\eta,
		$$
		where$$
		\begin{aligned}
			&M_{1j}=P_{j+1}^{u}(-G_{j}(v_{j})+F_{j}(v_{j})+0_{j+1}-F_{j}(0_{j})),  \\
			&M_{2j}=P_{j+1}^{u}(F_{j}(0_{j})-F_{j}(P_{j}^{s}v_{j})), \\
			&M_{3j}=P_{j+1}^{u}v_{j+1}.  \\
		\end{aligned}$$
		To estimate $ |M_{1j}|_{N} $, for one thing, by the definition of $ \|f - g\|_\infty $ and $ l_{j} $,
		$$ |G_{j}(v_{j})-F_{j}(v_{j})|_{N}\leqslant l_{j}^{-1}\|f - g\|_\infty\leqslant l_{j}^{-1}d\leqslant l_{j}^{-1}d_0. $$
		In addition, by the definition of the pseudo-orbit, $$\left|0_{j+1}-F\left(0_j\right)\right|_{N}
		\leq l^{-1}_j\delta_{0}.$$
		By (\ref{3.4}), $$ l_{j}^{-1}=\prod_{k=N_{i}}^{j-1}h_{k}^{-1}\leqslant R^{j-N_{i}} \leqslant C $$
		then by the choosing $ d_0 $ and $ \delta_{0} $, 
		\begin{align} \label{3.10}
			\begin{split}
				l_{j}^{-1}d_0+l_{j}^{-1}\delta_{0}
				\leqslant C\frac{1-\widetilde{\lambda}}{4C}\eta+C\cdot\frac{\delta_{1} }{C} 
				\leqslant   \frac{1-\widetilde{\lambda}}{4}\eta+\frac{1-\widetilde{\lambda}}{4}\eta=\frac{1-\widetilde\lambda}{2}\eta
			\end{split}
		\end{align}
		
		so
		\begin{align} \label{3.11}
			\begin{split}
				|M_{1j}|_{N}\leqslant|G_{j}(v_{j})-F_{j}(v_{j})|_{N}+|0_{j+1}-F_{j}(v_{j})|_{N}\leqslant\frac{1-\widetilde\lambda}{2}\eta.
			\end{split}
		\end{align}
		For $ |M_{2j}|_{N} $, this term arises from the difference between evaluating $ F_{j} $ at the origin and at the stable component $ P_{j}^{s}v_{j} $. Applying the Mean Value Theorem and (\ref{3.1}), (\ref{3.2}), we obtain:
		\begin{align} \label{3.12}
			\begin{split}
				|M_{2j}|_{N}
				&=|P_{j+1}^{u}(F_{j}(0_{j})-F_{j}(P_{j}^{s}v_{j}))|_{N}\\
				&=|\alpha_{j}^{u}(0_{j})-\alpha_{j}^{u}(P_{j}^{s}v_{j})|_{N}\\
				&=|\int_{0}^{1}\frac{\partial}{\partial s}\alpha_{j}^{u}(t\cdot P_{j}^{s}v_{j}) \cdot P_{j}^{s}v_{j}dt|_{N}\\
				&\leqslant \underset{\xi\in E_{j}^{u}(\eta) }{\sup}||\frac{\partial\alpha_{j}^{u}}{\partial s}(\xi)||_{N}\cdot |P_{j}^{s}v_{j}|_{N} \\
				&\leqslant   \left(\sup\limits_{\xi\in E^u_{j}(\eta)}
				\left\|\frac{\partial \alpha_j^u}{\partial s}(\xi) -\frac{\partial \alpha_j^u}{\partial s}(0) \right\|_{N}+ \left\|\frac{\partial \alpha_j^u}{\partial s}(0) -C_{j} \right\|_{N} +\|C_{j} \|_{N}\right) \eta\\
				&\leqslant (R\frac{\widetilde\lambda-\lambda}{4R}+R\frac{\widetilde\lambda-\lambda}{4R}  + R\frac{1+\lambda-2\widetilde\lambda}{4R}) \eta  \\
				&\leqslant \frac{1-\widetilde\lambda}{2} \eta.\\
			\end{split}
		\end{align}
		
		Furthermore, since $ \boldsymbol{v}\in\mathcal{B}_{j}^{u}(\eta) $, it is clear that
		\begin{align}\label{3.13}
			\begin{split}
				|M_{3j}|_{N}=||P_{j+1}^{u}v_{j+1}||\leqslant\eta.
			\end{split}
		\end{align}
		so from (\ref{3.11}),(\ref{3.12}) and (\ref{3.13}) , we conclude
		$$
		||M_{1j}+M_{2j}+M_{3j}||_{N}\leqslant \frac{1-\widetilde\lambda}{2}\eta+\frac{1-\widetilde\lambda}{2}\eta+\eta=(2-\widetilde\lambda)\eta\leqslant\widetilde\lambda^{-1}\eta.
		$$
		Hence, the operator is well-defined.
	\end{proof}
	\begin{lemma}\label{lem:3.4}
		The set  $ \mathcal{B}_{j}^{u}(\eta) $ is invariant under the operator $\mathcal{A}   $.
	\end{lemma}
	\begin{proof}[Proof of Lemma \ref{lem:3.4}]
		By Lemma \ref{lem:3.2} and inequality (\ref{3.8}), we have the uniform bound for the unstable component:
		$$
		|P_{j}^{u}w_{j}|_{N}\leqslant \widetilde\lambda^{-1}\eta\cdot\widetilde\lambda=\eta,\quad j=0,1,\cdots,N_{k+1}-1.
		$$
		To establish the invariance, it remains to show​ that the stable component also remains bounded by $ \eta $ i.e., $ |P_{j}^{s}w_{j}|_{N}\leqslant\eta $ in the following.
		
		For any given $ i\in [0,k] $ and $ j=N_{i},N_{i}+1,\cdots,N_{i+1}-1 $, we decompose the stable component at the next step as follows:
		$$
		P_{j+1}^{s}w_{j+1}=N_{1j}+N_{2j},
		$$
		where
		\begin{flalign}\label{3.14}
			\begin{split}
				&N_{1j}=P_{j+1}^{s}(G_{j}(v_{j})-F_{j}(v_{j})+F_{j}(0_{j})-0_{j+1}),  \\
				&N_{2j}=P_{j+1}^{s}(F_{j}(v_{j})-F_{j}(0_{j})). \\
			\end{split}
		\end{flalign} 
		To estimate $ |N_{1j}|_{N} $, by (\ref{3.8}),  we obtain directly
		\begin{align} \label{3.15}
			\begin{split}
				|N_{1j}|_{N}\leqslant|G_{j}(v_{j})-F_{j}(v_{j})|_{N}+|0_{j+1}-F_{j}(v_{j})|_{N}\leqslant\frac{1-\widetilde\lambda}{2}\eta.
			\end{split}
		\end{align}
		The term $ N_{2j} $ essentially measures the variation of the stable part of $ F_{j} $ from the origin. Using the Mean Value Theorem, we can write
		$$
		N_{2j}=\alpha^{s}(v_{j})-\alpha^{s}(0_{j})=\int_{0}^{1}(\frac{\partial}{\partial u}\alpha_{j}^{s}(tv_{j}),\frac{\partial}{\partial s}\alpha_{j}^{s}(tv_{j}))\cdot (v_{j}^{u},v_{j}^{s})^{T}dt.
		$$
		
		To bound this integral, we first estimate the partial derivatives of $ \beta_j^s $. By (\ref{3.1}), (\ref{3.2}) and the Mean Value Theorem, for any $\xi \in X_{j}(\eta)$,
		$$\begin{alignedat}{2}
			\left\|\frac{\partial}{\partial s} \alpha_j^s(\xi)\right\|_{N}&\leqslant \left\|\frac{\partial}{\partial s} \alpha_j^s(\xi)-\frac{\partial}{\partial s} \alpha_j^s(0) \right\|_{N} + \left\|\frac{\partial}{\partial s} \alpha																																	_j^s(0)-D_{j} \right\|_{N} + \|D_{j} \|_{N}\\
			&\leqslant R\frac{\widetilde\lambda-\lambda}{4R}+R\frac{\widetilde\lambda-\lambda}{4R}  + \lambda   \\
			&\leqslant \frac{\widetilde\lambda-\lambda}{4}+ \frac{\widetilde\lambda-\lambda}{4}+ \lambda\\
			&\leqslant \widetilde\lambda\end{alignedat}
		$$
		Similarly, for the other partial derivative,
		$$\begin{alignedat}{2}
			\left\|\frac{\partial}{\partial u} \alpha_j^s(\xi)\right\|_{N}& \leq
			\left\|\frac{\partial}{\partial u} \alpha_j^s(\xi)-\frac{\partial}{\partial u} \alpha_j^s(0) \right\|_{N}+\left\|\frac{\partial}{\partial u} \alpha_j^s(0)-C_{j} \right\|_{N}  +\|C_{j} \|_{N}\\
			&\leqslant R\frac{\widetilde\lambda-\lambda}{4R}+R\frac{\widetilde\lambda-\lambda}{4R}  + R\frac{1+\lambda-2\widetilde\lambda}{4R}\\
			&\leqslant \frac{\widetilde\lambda-\lambda}{4}+ \frac{\widetilde\lambda-\lambda}{4}+ \frac{1+\lambda-2\widetilde\lambda}{4}\\
			&\leqslant  \frac{1-\widetilde\lambda}{2} .
		\end{alignedat}
		$$
		Therefore, combining these derivative bounds, we obtain
		\begin{align} \label{3.16}
			\begin{split}
				|N_{2j}|_{N}
				&\leqslant \underset{t\in [0,1]}{\sup}||\frac{\partial}{\partial u}\alpha_{j}^{s}(tv_{j})||_{N}|v_{j}^{u}|_{N}+\underset{t\in [0,1]}{\sup}||\frac{\partial}{\partial s}\alpha_{j}^{s}(tv_{j})||_{N}|v_{j}^{s}|_{N} \\
				&\leqslant (\widetilde\lambda +\frac{1-\widetilde\lambda}{2} ) \left|v_j\right|_{N}\\
				&\leqslant \frac{1+\widetilde\lambda}{2}  \eta \\
			\end{split}
		\end{align}
		
		Finally, combining the estimates (\ref{3.15}), (\ref{3.16}), we conclude that
		$$
		|P_{j+1}^{s}w_{j+1}|_{N}\leqslant |N_{1j}|_{N}+|N_{2j}|_{N}\leqslant\frac{1-\widetilde\lambda}{2}\eta+\frac{1+\widetilde\lambda}{2}  \eta=\eta.
		$$
	\end{proof}
	{\em Step 4.} Complete the proof of
	Theorem \ref{th:3.1}. 
	
	By Lemma \ref{lem:3.4} and Brouwer Fixed Point Theorem, there is a fixed point $ \boldsymbol{v}=\left\{v_j\right\}_{j=0}^{N_{k+1}} $ of $ \mathcal{A} $ in $ \mathcal{B}_{j}^{u}(\eta) $. To complete the proof of Theorem, it remains to verify that this fixed point satisfies the recurrence relation  $ v_{j+1}=G_{j}(v_{j}) $ for $ j=0,1,\cdots,N_{k+1}-1. $
	
	Since $ \boldsymbol{v} $ is a fixed point of the operator $ \mathcal{A} $, so equation (\ref{3.6}), (\ref{3.7}), (\ref{3.8})can be rewritten as:
	
	For $ j=0,N_{k+1}-1, $
	\begin{align}\label{3.17}
		P_0^{s}v_{0}=0, \quad
		P_{N_{k+1}}^{u}v_{N_{k+1}}=0.
	\end{align}
	
	For $ j=0,\cdot\cdot\cdot,N_{k+1}-1, $
	\begin{align}\label{3.18}
		P_{j+1}^{s}v_{j+1}=P_{j+1}^{s}(G_{j}(v_{j})).
	\end{align}
	
	For $ j=0,\cdot\cdot\cdot,N_{k+1}-1,$
	\begin{align}\label{3.19}
		P_{j}^{u}v_{j}=\Psi_{v_{j}}(P_{j+1}^{u}(-G_{j}(v_{j})+F_{j}(v_{j})-F_{j}(P_{j}^{s}v_{j})+v_{j+1})), \quad j=0,\cdot\cdot\cdot,N_{k+1}-1.
	\end{align}

	Applying the nonlinear operator $ \Phi_{v_{j}}=\Psi_{v_{j}}^{-1} $ to both sides of this last equation, obtain
	$$
	\Phi_{v_{j}}(P_{j}^{u}v_{j})=P_{j+1}^{u}(-G_{j}(v_{j})+F_{j}(v_{j})-F_{j}(P_{j}^{s}v_{j})+v_{j+1}).
	$$
	Recall the definition of $ \Phi_{v_{j}} $,
	$$
	\Phi_{v_{j}}(P_{j}^{u}v_{j})=P_{j+1}^{u}(F_{j}(P_{j}^{s}v_{j}+P_{j}^{u}v_{j})-F_{j}(P_{j}^{s}v_{j})).
	$$
	Comparing the last two equations and using the linearity of the projection $ P_{j+1}^{u} $ gives 
	$$ P_{j+1}^{u}(-G_{j}(v_{j})+v_{j+1})=0. $$
	Consequently,
	\begin{flalign}\label{3.20}
		\begin{split}
			P_{j+1}^{u}v_{j+1}= P_{j+1}^{u}(G_{j}(v_{j}))
		\end{split}
	\end{flalign}
	Equation (\ref{3.18}) already provides the stable component:
	$$
	P_{j+1}^{s}v_{j+1}=P_{j+1}^{s}(G_{j}(v_{j})).
	$$
	Together with (\ref{3.20}) we obtain the equality of the full vectors:
	\begin{equation}\label{3.21}
		v_{j+1}=G_{j}(v_{j}),\quad j=0,1,\cdots,N_{k+1}-1.
	\end{equation}
	Define $ x=\exp_{y_0}(v_0)$. From (\ref{3.21}) and the definition of the maps $ G_{j} $ it follows inductively that
	$$
	g^{j}x=\exp_{y_j}(v_j).
	$$
	Hence,
	$$
	\rho(g^{j} x,f^{j-N_{i}}x_i)=\rho(\exp_{y_j}(v_j),y_j)=||v_j||=l_{j}|v_j|_{N}\leqslant \eta \leqslant \varepsilon_1.
	$$
	which completes the proof of Theorem \ref{th:3.1}.
\end{proof}

\hfuzz=1pt
\section{Bi-Shadowing of Infinite Quasi-semi-hyperbolic Pseudo-orbit and Periodic Bi-Shadowing}\label{Section 4}

In this section, we will show that infinite quasi-semi-hyperbolic pseudoorbit has the  bi-shadowing property and periodic bi-shadowing property. 

\begin{theorem}\label{th:4.1}
	Let $f \in \operatorname{Diff}(M),0 < \lambda < 1 $, and let $TM=E^u \oplus E^s$ be a continuous splitting. For a given $ \varepsilon_1>0 $, then there exist three positive numbers   $\delta_0, \varepsilon_0$ and $d_0$ with the following property: For any $g \in \operatorname{Diff}(M)$ satisfying $ \|f - g\|_\infty\leqslant d \leqslant d_0 $, and for any  $(\lambda,\varepsilon,\delta)$-quasi-semi hyperbolic pseudo-orbit $  \{x_{i},n_{i}\}_{i=-\infty}^{\infty} $ of $ f $ with $ \delta\leqslant\delta_0,\varepsilon\leqslant\varepsilon_0 $, there exists a point $ x\in M $ that $ \varepsilon $-bi-shadows $ \{x_{i},n_{i}\}_{i=-\infty}^{\infty}$ i.e., $ \rho(g^j x,f^{j-N_{i}}x_{i})\leqslant\varepsilon_1 $ for $ N_{i}\leqslant j\leqslant N_{i+1}-1 $.
\end{theorem}

\begin{proof}
	For the given infinite $(\lambda,\varepsilon,\delta)$-quasi-semi hyperbolic pseudo-orbit $  \{x_{i},n_{i}\}_{i=-\infty}^{\infty} $, denote the points $  y_{j}=f^{j-N_{i}}x_{i}$ for $ N_{i}\leqslant j\leqslant N_{i+1}-1.$ Without loss of generality we may restrict our attention to indices $ i\in [0,k] $ for an arbitrarily large $ k $; the construction will be extended to all indices by a limiting argument. 
	
	For each fixed $ k $, consider the finite​ segment $$ \{y_{0}^{(k)},y_{1}^{(k)},\cdots,y_{N_{k+1}}^{(k)}\}\quad \text{with} \quad y_{j}^{(k)}=y_{j}\quad j=0,1,\cdots,N_{k+1}.$$  
	Given $\lambda \in (0,1)$, Theorem \ref{th:3.1}, whose proof was given earlier, guarantees the existence of constants $\delta_0, \varepsilon_0$ and $d_0$ such that the following holds: For any nonnegative numbers $ \delta\leqslant\delta_0,\varepsilon\leqslant\varepsilon_0,d\leqslant d_0,$ and any $g \in \operatorname{Diff}(M)$ satisfying $ \|f - g\|_\infty\leqslant d $, there exists a point $ \boldsymbol{v}^{(k)}=\left\{v_j^{(k)}\right\}_{j=0}^{N_{k+1}} $ of $ \mathcal{A} $ in $ \mathcal{B}_{j}^{u}(\eta) $ such that 
	\begin{equation}\label{4.1}v_{j+1}^{(k)} = G_j(v_j^{(k)}),\quad j=0,1,\cdots,N_{k+1}-1.\end{equation}
	
	Fix $j \in \mathbb{Z}$. For all sufficiently large $ k $, the points $ v_j^{(k)} $ are defined and belong to the bounded set $ \mathcal{B}_{j}^{u}(\eta) $. Hence the sequence $\{v_j^{(k)}\}_{k \in \mathbb{Z}}$ is bounded and thus, without loss of generality, by taking an appropriate subsequence and relabeling, suppose that it converges to a limit $v_j$. That is, $\lim_{k \to \infty} v_j^{(k)} = v_j.$
	
	Taking the limit $k \to \infty$ on both sides of (\ref{4.1}) and using the continuity of the maps $ G_j $, we obtain $v_{j+1}=G_j\left(v_j\right), j \in \mathbb{Z}$. Define $ x=\exp_{y_0}(v_0)$. By the definition of the maps $ G_{j} $ it follows inductively that
	$$
	g^{j}x=\exp_{y_j}(v_j).
	$$
	Hence,
	$$
	\rho(g^{j} x,f^{j-N_{i}}x_i)=\rho(\exp_{y_j}(v_j),y_j)=||v_j||=l_{j}|v_j|_{N}\leqslant \eta \leqslant \varepsilon_1.
	$$
	which completes the proof of Theorem \ref{th:3.1}.
\end{proof}

Periodic behavior is often of particular interest in dynamical systems. Here it is 
shown that bi-shadowing is preserved when attention is restricted to periodic trajectories, also often called cycles.

\begin{definition}
    A $(\lambda,\varepsilon,\delta)$-quasi-semi hyperbolic pseudo-orbit $  \{x_{i},n_{i}\}_{i=-\infty}^{\infty} $ of $ f $ is periodic, i.e., there exists an $ m, $ such that $ x_{i+m}=x_{i} $ and $ n_{i+m}=n_{i} $ for all $ i $. We say $ f $ has the periodic bi-shadowing property if the shadowing point $  x $ can be chosen to be periodic with period $ N_{m} $.
\end{definition}

\begin{theorem}\label{th:4.2}
	Let $f \in \operatorname{Diff}(M),0 < \lambda < 1 $, and let $TM=E^u \oplus E^s$ be a continuous splitting. For a given $ \varepsilon_1>0 $, then there exist three positive numbers   $\delta_0, \varepsilon_0$ and $d_0$ with the following property: For any $g \in \operatorname{Diff}(M)$ satisfying $ \|f - g\|_\infty\leqslant d \leqslant d_0 $, and for any periodic $(\lambda,\varepsilon,\delta)$-quasi-semi hyperbolic pseudo-orbit $  \{x_{i},n_{i}\}_{i=-\infty}^{\infty} $ of $ f $ with $ \delta\leqslant\delta_0,\varepsilon\leqslant\varepsilon_0,x_{i+m}=x_{i},n_{i+m}=n_{i},m>0 $, then $ f $ has the periodic bi-shadowing property, i.e., $$ \rho(g^j x,f^{j-N_{i}}x_{i})\leqslant\varepsilon_1\quad \text{and} \quad g^{N_{m}}x=x $$ for $ N_{i}\leqslant j\leqslant N_{i+1}-1 $.
\end{theorem}
\begin{proof}
	Assume that $  \{x_{i},n_{i}\}_{i=-\infty}^{\infty} $ is a periodic $(\lambda,\varepsilon,\delta)$-quasi-semi hyperbolic pseudo-orbit of $ f $ with $ \delta\leq\delta_0,\varepsilon\leq\varepsilon_0,x_{i+m}=x_{i},n_{i+m}=n_{i},m>0 $. Denote the points $  y_{j}=f^{j-N_{i}}x_{i}$ for $ N_{i}\leqslant j\leqslant N_{i+1}-1,$ then $y_{N_{m}}=y_0 $. 
	The assumptions for the proof below are essentially the same as those of Theorem \ref{th:3.1}. We construct the same operator $ \mathcal{A} $ on the space $$
	\mathcal{B}_{j}^{u}(\eta)=\{\{v_{j}\}_{j=0}^{N_{m}}\in E_{j}^{u}:|v_{j}|_{*}\leqslant\eta\}.
	$$
	but with one essential change: the boundary conditions​ are altered to reflect the periodicity.
	
	In (\ref{3.4}),	the boundary conditions for the operator $ \mathcal{A} :$
	$$
	P_{0}^{s}w_{0}=0 \quad \text{and} \quad P_{N_{k+1}}^{u}w_{N_{k+1}}=0
	$$
	For the periodic case we replace them by the periodic conditions 
	\begin{flalign}\label{4.2}
		\begin{split}
			P_{0}^{s}w_{0}=P_{N_{m}}^{s}w_{N_{m}} \quad \text{and} \quad P_{N_{m}}^{u}w_{N_{m}}=P_{0}^{u}w_{0}
		\end{split}
	\end{flalign}
	All other formulas defining $ \mathcal{A} $ remain exactly as in Theorem \ref{th:3.1}.
	Then, following a proof similar to that of Theorem \ref{th:3.1}, we obtain a fixed point  $ \boldsymbol{v}=\left\{v_j\right\}_{j=0}^{N_{m}} $ of $ \mathcal{A} $ in $ \mathcal{B}_{j}^{u}(\eta) $. Then (\ref{4.2}) can be written as 
	$$
	P_{0}^{s}v_{0}=P_{N_{m}}^{s}v_{N_{m}} \quad \text{and} \quad P_{N_{m}}^{u}v_{N_{m}}=P_{0}^{u}v_{0}
	$$
	combining 
	$$
	v_{0}=P_{0}^{s}v_{0}+P_{0}^{u}v_{0} \quad \text{and} \quad v_{N_{m}}=P_{N_{m}}^{u}v_{N_{m}}+P_{N_{m}}^{s}v_{N_{m}},
	$$
	we get $ v_{0}=v_{N_{m}} $. We now carry out the following steps, which closely follow the concluding argument of Theorem \ref{th:3.1}. Define $ x=\exp_{y_0}(v_0)$. Hence,
	$$
	g^{N_{m}}x=\exp_{y_{N_{m}}}(v_{N_{m}})=\exp_{y_0}(v_0)=x.
	$$ Thus $ x $ is a periodic point of $ g $ with period $ N_{m} $ which completes the proof of Theorem \ref{th:4.2}.
\end{proof}

\newpage
\addcontentsline{toc}{section}{References}

\end{document}